\documentclass[11pt, reqno, a4paper, final]{amsart}

\usepackage[applemac]{inputenc}
\usepackage[english]{babel}
\usepackage[T1]{fontenc}
\usepackage{amsmath,amsfonts,amssymb,amsthm,amscd}
\usepackage{eucal}
\usepackage{array}
\usepackage{color}
\usepackage{hyperref}
\usepackage{a4wide}										 	
\usepackage{mathrsfs}									  
\usepackage[all]{xy}										
\usepackage{url}												
\usepackage{verbatim,epsfig}
\usepackage[notref,notcite,draft]{showkeys}   
\usepackage{xfrac} 

\theoremstyle{plain}

\newtheorem{thm}{Theorem}[section] 
 
\newtheorem{cor}[thm]{Corollary}

\newtheorem{lem}[thm]{Lemma}

\newtheorem{prop-defi}[thm]{Definition \& Proposition}
\newtheorem{prop}[thm]{Proposition}

\newtheorem*{thm*}{Theorem}
\newtheorem*{prop*}{Proposition}
\newtheorem*{cor*}{Corollary}

\theoremstyle{definition}
\newtheorem{defi}[thm]{Definition}

\newtheorem{rem}[thm]{Remark}

\newtheorem{claim}{Claim}
\newtheorem*{claim*}{Claim}

\newcommand{\NN}{{\mathbb N}}

\newcommand{\RR}{{\mathbb R}}

\newcommand{\D}{\mathcal{D}}

\newcommand{\partafl}[2]{\frac{\partial {#1}}{\partial {#2}}}

\newcommand{\varps}{{\varepsilon}}
\newcommand{\rg}{{\operatorname{rg\hspace{0.04cm}}}}

\newcommand{\tens}{\otimes}

\newcommand{\To}{\longrightarrow}

\newcommand{\Hom}{\operatorname{Hom}}

\newcommand{\id}{\operatorname{id}}

\newcommand{\Ad}{\operatorname{Ad}}

\newcommand{\g}{\mathfrak g}

\renewcommand{\leq}{\leqslant}
\renewcommand{\geq}{\geqslant}

\newcommand{\im}{{\operatorname{im}}}

 \newcommand{\mfg}{{\mathfrak{g}}}
\newcommand{\mfh}{{\mathfrak{h}}}
\newcommand{\mfk}{{\mathfrak{k}}}
\newcommand{\Ug}{{\mathcal{U}(\mathfrak{g})}}
\newcommand{\Uh}{{\mathcal{U}(\mathfrak{h})}}

\newcommand{\Cohom}{{\operatorname{H}}}

\newcommand{\Coind}{{\operatorname{Coind}}}

\newcommand{\End}{\operatorname{End}}
\newcommand{\Vect}{\operatorname{Vect}}

\newcommand{\ad}{\operatorname{ad}}

\renewcommand{\restriction}{\rvert}
\title{Topologizing Lie algebra cohomology}

\author{David Kyed}

\address{David Kyed, Department of Mathematics and Computer Science, University of Southern Denmark, Campusvej 55, DK-5230 Odense M, Denmark}
\email{dkyed@imada.sdu.dk}

\begin{document}

\begin{abstract}
We show that the theory of  Lie algebra cohomology can be recast in a topological setting and that classical results, such as the Shapiro lemma and the van Est isomorphism, carry over to this augmented context.  
\end{abstract}
\maketitle

\section{Introduction}

The theory of cohomology for groups and Lie algebras dates  back to the pioneering works of, among others,  Cartan, Chevalley, Eilenberg, Kozul and Mac Lane  \cite{CE, chevalley-eilenberg, koszul},  and is by now an indispensable tool in a variety of different branches of mathematics. In recent years, there has been an increasing interest in the topological aspects of group cohomology, since it turns out that there are many instances where one does not have vanishing of the group cohomology on the nose, but only  of the \emph{reduced} cohomology (see e.g.~\cite{ shalom-rigidity, bader-furman-sauer, tessera-vanishing} for examples of this phenomenon).
 When the group $G$ in question is a connected  Lie group and the coefficient module is smooth, the van Est theorem provides an (a priori algebraic) isomorphism between the cohomology of $G$  and the (relative) Lie algebra cohomology of its Lie algebra, and keeping in mind the abundance of results involving \emph{reduced} group cohomology, it is natural to ask if also the Lie algebra cohomology carries a canonical topology and, if so, whether or not the van Est isomorphism is actually a homeomorphism. Both questions are answered affirmatively in Section \ref{top-lie-alg-comhom-section} and Section \ref{top-van-est-thm-section}, respectively.  Along the way, we provide a reasonably self contained introduction to the theory of Lie algebra cohomology, with the hope of making our results more accessible to non-experts. It should also be mentioned that the topological van Est theorem was already stated, and used, in \cite{polynomial-cohomology} and is an important tool in our ongoing project concerning polynomial cohomology of nilpotent Lie groups. Throughout the paper, emphasis will be put on the arguments pertaining to Lie algebra cohomology and even though cohomology of Lie groups is a central theme in Section \ref{top-van-est-thm-section}, we will assume familiarity with this theory (although references will be given whenever appropriate) which can be found in \cite[Chapter III]{guichardet-book} or \cite[Chapter X]{borel-wallach}. We will develop the cohomology theory for Lie algebras within the framework of relative homological algebra,  primarily following \cite{guichardet-book},   and in many cases the passage from the algebraic context to the topological one merely consists of making sure that all maps involved respect the topologies.


\subsubsection*{Acknowledgements} The author gratefully acknowledges the financial support from the Lundbeck foundation (grant R69-A7717) and the Villum foundation (grant 7423).  Furthermore, thanks are certainly due to Henrik Densing Petersen, whose idea it was to topologize the Lie algebra cohomology and extend the van Est theorem to the topological setting.

\section{Topological Lie algebra cohomologi}\label{top-lie-alg-comhom-section}
Throughout this section, $\mfg$ denotes a (finite dimensional) Lie algebra over the reals and $\mfh$ denotes a Lie sub-algebra of $\mfg$. For now, there are no restrictions on $\mfh$, but in order to develop the cohomology theory of $\mfg$ relative to $\mfh$ we will soon  require $\mfh$ to be reductive; see section \ref{standard-reso-subsec} and Remark \ref{reductive-rem}.

\begin{defi}
A continuous (or topological) $\mfg$-module is a Hausdorff topological vector space (t.v.s.) $E$ with an action of the Lie algebra $\mfg$ such that each $X\in \mfg$ acts as a continuous operator.  A morphism of continuous $\mfg$-modules (also simply referred to as a $\mfg$-morphism) is a continuous, linear map of t.v.s.~that intertwines  the $\mfg$-actions. An element $\xi\in E$ is said to be $\mfg$-invariant if $X.\xi=0$ for all $X\in \mfg$ and the set of $\mfg$-invariant elements is denoted $E^\mfg$.
\end{defi}
By the universal property of  the enveloping algebra $\Ug$, any Lie algebra representation extends to an algebra representation of $\Ug$, and if the representation of $\mfg$ is by continuous operators on a t.v.s., then so is the induced representation of $\Ug$. In what follows, we will freely identify the representation of $\mfg$ with the corresponding representation of $\Ug$.  The key to getting homological algebra working in this topological context is to pin-point the right definition of morphisms and injective modules, which we will adapt, mutatis mutandis, from from the corresponding theory for groups \cite[Chapter III]{guichardet-book}:

\begin{defi}
Let $E$ and $F$ be continuous  $\mfg$-modules. An injective $\g$-morphism $f\colon E\to F$ is said to be \emph{$\mfh$-strengthened} if there exists a continuous $\mfh$-equivariant map $s\colon F\to E$ such that $s\circ f=\id_E$. A general $\mfg$-morphism $f\colon E\to F$ is said to be $\mfh$-strengthened if both the inclusion $\ker(f) \to E$ and the induced map $E/\ker(f) \to F$ are $\mfh$-strengthened in the sense just defined. Lastly, a $\mfg$-morphism $f\colon E\to F$ is called \emph{strengthened} if is $\mfh$-strengthened with respect to the trivial  subalgebra.
\end{defi}

\begin{defi}
A continuous $\mfg$-module $E$ is called \emph{$\mfh$-relative injective} if for any two other such modules $A, B$, any $\mfh$-strengthened injective $\mfg$-morphism $\iota\colon A\to B$ and any $\mfg$-morphism $f\colon A\to E$ there exists a $\mfg$-morphism $\tilde{f}\colon B\to E$ such that $f=\tilde{f}\circ \iota$. When the subalgebra $\mfh$ is the trivial one, we simply refer to $\mfh$-relative injective modules as being \emph{relative injective}.
\end{defi}

One may now consider  $\mfh$-strengthened, $\mfh$-relative injective resolutions of a given $\mfg$-module $E$ and an adaptation of the standard arguments from homological algebra (carried out in detail in Appendix \ref{rel-hom-alg-appendix}), implies that given any two such resolutions, upon passing to $\mfg$-invariants and thereafter to cohomology, the resulting cohomology spaces are isomorphic in each degree --- the isomorphism being as (generally non-Hausdorff!) topological vector spaces. Thus, if we can show that any continuous $\mfg$-module $E$ admits such a resolution, then the cohomology of $\mfg$, relative to $\mfh$, with coefficients in $E$ is well defined as a topological object; in what follows we provide such a resolution under the mild additional assumption that  $E$ is locally convex.

\subsection{Morphism spaces}

Let $E$ and $F$ be Hausdorff t.v.s.~and consider the space $\Hom(E,F)$ of continuous linear maps. On $\Hom(E,F)$ we will consider the topology of \emph{uniform convergence on compacts}, which is defined by the following family of neighbourhoods of zero, indexed by the compact subsets $K\subset E$ and open zero-neighbourhoods $V\subset F$:
\[
N_{K,V}:=\{T\in \Hom(E,F) \mid T(K)\subset V\}.
\]
Note that when the topology on $F$ is locally convex, and thus  generated by a family of semi-norms $(p_i)_{i\in I}$, then the family of seminorms
\[
p_{K,i}(T):=\sup_{\xi \in K} p_i(T\xi), \quad i\in I, K\subset E \text{ compact},
\]
generates the topology on $\Hom(E,F)$. In this case, the topology of uniform convergence on compacts is therefore a locally convex, Hausdorff vector space topology as well.
We shall be particularly interested in the case when $E$ is the enveloping algebra $\Ug$; recall that $\Ug$ is the quotient of the tensor algebra $\mathcal{T}(\mfg)$ by the ideal generated by elements of the form $x\tens y- y\tens x-[x,y]$. As $\mathcal{T}(\mfg)$ is the increasing union of the finite dimensional subspaces $ \oplus_{k=0}^n \mfg^{\tens k} $, by pushing these spaces down to $\Ug$, we obtain a strictly increasing family of finite dimensional subspaces $(V_n)_{n\in \NN}$ whose union is all of $\Ug$. On $\Ug$ we may therefore consider the inductive limit topology \cite[II,6]{schaefer} defined by the filtration $(V_n)_n$; recall that this is the finest locally convex topology on $\Ug$ such that all the inclusions $V_n\hookrightarrow \Ug$ are continuous.  Note also, that a linear map $T\colon\Ug\to F$ into some locally convex space $F$ is continuous iff $T\restriction_{V_n}$ is continuous for all $n$ \cite[II 6.1]{schaefer}, and since $V_n$ is finite dimensional this is automatic. This shows the following:
 
\begin{lem}
The inductive limit topology on $\Ug$ defined by the subspaces $(V_n)_n$ coincides with the maximal locally convex topology; i.e.~the one defined by declaring that all possible seminorms on $\Ug$ be continuous. 
\end{lem}
In particular, this implies that multiplication  from the left (or right) by a fixed element in $ \Ug$ is continuous. 



\begin{lem} On $\Hom(\Ug,F)$ the topology of uniform convergence on compacts coincides with the topology of pointwise convergence or, equivalently, the topology of pointwise convergence on any linear basis of $\Ug$.
In particular, $\Hom(\Ug, F)$ is a Fr{\'e}chet space whenever $F$ is.
\end{lem}
\begin{proof}
Since points in $\Ug$ are compact, uniform convergence on compacts clearly implies pointwise convergence. Assume, conversely, that $f_i\in\Hom(\Ug,F)$ is a net converging pointwise to zero and let $K\subset \Ug$ be a compact set. Since  compact sets are (totally) bounded \cite[I.5]{schaefer} and $\Ug$ is endowed with the strict inductive limit topology arising from the finite dimensional subspaces $V_n$, we conclude from \cite[II.6.4 \& 6.5]{schaefer} that there exists an $n_0$ such that $K\subset V_{n_0}$, and that $K$ is compact as a subset of $V_{n_0}$. Thus,  if $e_1,\dots, e_p$ is a linear basis for $V_{n_0}$ there exists an $R>0$ such that
\[
K\subset \left\{  \sum_{i=1}^p t_i e_i \ \Big{\rvert} \ |t_i| \leq R \right\}.
\] 
Now fix a zero-neighbourhood $V\subset F$ and choose a smaller zero-neighbourhood $V'\subset F$ such that $V'+V'+ \cdots +V'\subset V$ ($p$ summands). Since the scalar action on $F$ is continuous there exists a zero neighbourhood $V''$ such that $tV''\subset V'$ whenever $|t|\leq R$. Lastly, since $f_i$ is assumed to converge pointwise to zero, there exists an $i_0$ such that $f_i(e_k)\in V''$ for all $i\geq i_0$ and all $k=1,\dots, p$, and hence $f_i(K)\subset V$ for all $i\geq i_0$ as desired. Clearly pointwise convergence and pointwise convergence on a linear basis is the same as we are dealing with linear maps. Assume now that $F$ is Fr{\'e}chet and choose a countable family of seminorms $(p_i)_{i\in \NN}$ and a countable linear basis $(e_l)_{l\in \NN}$ for $\Ug$. Then the topology of pointwise convergence is generated by the countable family of seminorms
$
p_{i,l}(f):=p_i(f(e_l)),
$
and the completeness of $F$ translates into  completeness of $\Hom(\Ug,F)$ for the topology of pointwise convergence;  hence $\Hom(\Ug,F)$ is a Fr{\'e}chet space.
\end{proof}

\begin{rem}
In what follows, we will also be considering spaces of the form $\Hom(\Ug\tens V  , F)$ for some finite dimensional space $V$. 
Since $V$ is finite dimensional, $\Ug\tens V$ is nothing but a finite amplification of $\Ug$ and we will always consider $\Ug\tens V$ with the direct sum topology topology arising from the identification with a finite direct sum of copies of $\Ug$.

\end{rem}

\subsection{The standard resolution}\label{standard-reso-subsec}
Consider again a Lie algebra $\mfg$ and a reductive subalgebra $\mfh\leq \mfg$. Recall that $\mfh$ is reductive if it splits as a direct sum of a semi-simple and an abelian Lie algebra, and that this is the case whenever $\mfh$ is the Lie algebra of a compact group  --- this will be the main case of interest to us in connection with the van Est theorem.  For $X\in \mfg$ we denote by $\bar{X}$ the class of $X$ in $\mfg/\mfh$. Let $E$ be a locally convex,  continuous $\mfg$-module and denote by $E^n$ the subspace of $\Hom\left(\Ug\tens \wedge^n\mfg/\mfh, E\right)$ consisting of those $f$ satisfying
\begin{align}\label{h-invariance}
Y.f(u, \bar{X}_1, \dots, \bar{X}_n)- f(Yu, \bar{X}_1,\dots, \bar{X}_n) - \sum_{i=1}^n f(u, \bar{X}_1, \dots, \overline{[Y,X_i]},\dots, \bar{X}_n )=0
\end{align}
Here, and in what follows, we think of $\Hom\left(\Ug\tens \wedge^n\mfg/\mfh, E\right)$ as the set of multilinear maps from the cartesian product $\Ug\times \prod_{i=1}^n \mfg/\mfh $ to $E$ which are antisymmetric an all but the first variables, which this makes sense since $E$ is assumed locally convex and hence any linear map from $\Ug\tens \wedge^n\mfg/\mfh$ to $E$ is continuous. Since $E$ is a continuous $\mfg$-module, $E^n$ is a closed subspace of $\Hom\left(\Ug\tens \wedge^n\mfg/\mfh, E\right)$ and since multiplication by $X\in \mfg$ is in $\Hom(\Ug,\Ug)$, $E^n$ may be turned into a continuous $\mfg$-module by setting
\[
(X.f) (u,\bar{X}_1,\dots, \bar{X}_n):= f(uX, \bar{X}_1,\dots, \bar{X}_n).
\]
Now define $d^n\colon E^n\to E^{n+1}$ by the formula\footnote{As is standard,  $\hat{\bar{X}}_i$ here means that $\bar{X}_i$ is left out.}
\begin{align*}
d^n(f)(u,\bar{X}_1,\dots, \bar{X}_{n+1}) &: = \sum_{i=1}^{n+1} (-1)^{i+1} X_if(u,\bar{X}_1,\dots, \hat{\bar{X}}_i, \dots, \bar{X}_{n+1})+\\
&+ \sum_{i=1}^{n+1}(-1)^{i} f({X}_i u,\bar{X}_1,\dots, \hat{\bar{X}}_i,\dots, \bar{X}_{n+1})+\\
&+ \sum_{i<j} f(u,\overline{[X_i,X_j]},\bar{X}_1,\dots, \hat{\bar{X}}_i, \dots, \hat{\bar{X}}_j,\dots, \bar{X}_{n+1}).
\end{align*}
and define $\varps\colon E\to E^0:=\Hom(\Ug,E)$ by $\varps(\xi)(u):=u.\xi$. Note that the definition of $d^n$ makes sense  because of the imposed invariance \eqref{h-invariance}.  Moreover, it is clear that the maps $d^n$ are continuous and one may show that $d^0\circ \varps=0$ and $d^{n+1}\circ d^n=0$ for $n\geq 0$ \cite[Chapter II]{guichardet-book}.  Our aim now is to prove that
the complex
\begin{align}\label{standard-reso}
0\To E\overset{\varps}{\To} E^0 \overset{d^0}{\To} E^1 \overset{d^1}{\To} E^2 \overset{d^2}{\To} \cdots
\end{align}
is a $\mfh$-strengthened, $\mfh$-relative injective resolution of the continuous $\mfg$-module $E$; the resolution \eqref{standard-reso} will be referred to as the \emph{the standard resolution} of $E$ in the sequel.
Actually, the proof of this involves passing to an isomorphic complex which we will now describe. Denote by $\tilde{E}^n$ the subspace of $\Hom(\Ug\tens \wedge^n\mfg/\mfh, E)$ satisfying
\[
f(uY, \bar{X}_1, \dots, \bar{X}_n)=\sum_{i=1}^nf(u, \bar{X}_1,\dots, \overline{[Y,X_i]}, \dots, \bar{X}_n)
\]
for all $Y\in \mfh$, $u\in \Ug$ and $\bar{X}_1, \dots, \bar{X}_n\in \mfg/\mfh$ and endow it with a continuous $\mfg$-action by setting
\begin{align}\label{tilde-g-action}
(X.f)(u,\bar{X}_1,\dots, \bar{X}_n):= X.f(u,\bar{X}_1,\dots, \bar{X}_n)-f(Xu,\bar{X}_1,\dots, \bar{X}_n).
\end{align}
Furthermore define $\tilde{\varps}\colon E\to \tilde{E}^0$ and $d^n\colon \tilde{E}^n\to \tilde{E}^{n+1}$ by setting
\begin{align}\label{tilde-coboundary-map}
\tilde{d}^n(f)(u, \bar{X}_1,\dots, \bar{X}_{n+1})&:=\sum_{i=1}^{n+1} (-1)^{i+1}f(uX_i, \bar{X}_1,\dots, \hat{\bar{X}}_i,\dots, \bar{X}_{n+1}) + \notag\\
& + \sum_{i<j}(-1)^{i+j} f(u, \overline{[X_i,X_j]}, \bar{X}_1,\dots, \hat{\bar{X}}_i, \dots, \hat{\bar{X}}_j, \dots, \bar{X}_{n+1}),\notag\\
\tilde{\varps}(\xi)(u)&:=\eta(u).\xi,
\end{align}
where $\eta\colon \Ug\to \RR$ is the augmentation map. The claim now is the following:
\begin{prop}\label{the-two-isomorphic-resolutions}
The complexes 
\begin{align}
&0\To E   \overset{\varps}{\To } E^0  \overset{d^0}{\To} E^1 \overset{d^1}{\To }  E^2  \overset{d^2}{\To }\cdots \label{E-reso}\\
& 0\To E   \overset{\tilde{\varps}}{\To } \tilde{E}^0  \overset{\tilde{d}^0}{\To} \tilde{E}^1 \overset{\tilde{d}^1}{\To }  \tilde{E}^2  \overset{\tilde{d}^2}{\To }\cdots \label{E-tilde-reso}
\end{align}
are isomorphic as complexes of topological $\mfg$-modules and are  $\mfh$-strengthened, $\mfh$-relative injective resolutions of $E$.
\end{prop}
The complex \eqref{E-tilde-reso} will be referred to as the \emph{twisted standard resolution} below. Strictly speaking we do now know that  \eqref{E-tilde-reso}  is a complex at this point, but this will  follow from the proof of Proposition \ref{the-two-isomorphic-resolutions}.
\begin{proof}
The strategy is to provide a bicontinuous isomorphism between the two complexes, prove that the complex \eqref{E-tilde-reso} is a $\mfh$-strengthened resolution and that each $E^n$ is $\mfh$-relative injective. To see that \eqref{E-tilde-reso} is strengthened, one realizes it (algebraically) as the dual of a certain homological complex $(C_n,d_n)$ which is proven to be exact and (algebraically) $\mfh$-strengthened in \cite[Lemme 2.2 \& 2.3]{guichardet-book}. This homological complex therefore admits a  $\mfh$-linear contracting homotopy  $(s_n)_n$ and dualizing this we obtain a $\mfh$-linear contracting homotopy $(s^n)_n$ for the complex \eqref{E-tilde-reso}. But since $s^n(f):=f\circ s_n$ and the topology on $\tilde{E}^n\subseteq \Hom(\Ug\tens \wedge^n \mfg/\mfh, E)$ is the topology of pointwise convergence, it is clear that the maps $s^n$ are automatically continuous. This proves that \eqref{E-tilde-reso} is strengthened (see e.g.~Proposition \ref{strongly-exact-gives-contractible} and Remark \ref{converse-rem}).  To see that $E^n$ is indeed $\mfh$-relative injective, consider the setup for $\mfh$-relative injectivity; that is a diagram of the form
\[
\xymatrix{
0 \ar[r] & A \ar[d]_f \ar[r]_\iota & B  \ar@/_/@{->}[l]_{s} \ar@{-->}[dl]^{\exists? \tilde{f}}  \\
 & E^n &
}
\]
in which $\iota$ and $f$ are $\mfg$-morphisms and $s$ is a continuous $\mfh$-linear map witnessing the fact that $\iota$ is assumed strengthened.
We then define $\tilde{f}$ by setting 
\[
\tilde{f}(b)(u,\bar{X}_1,\dots ,\bar{X}_n):= f(s(u.b))(1,\bar{X}_1,\dots, \bar{X}_n).
\]
A direct computation shows that $\tilde{f}$ is $\g$-linear and that $\tilde{f}\circ \iota=f$,
 so we only need to prove that $\tilde{f}$ is continuous. 
So, let $b_j\to_j 0$ in $B$ and let $u\in \Ug$ and $\bar{X}_1, \dots, \bar{X}_n \in \mfg/\mfh $ be given. Then since $B$ is a continuous $\mfg$-module and both $f$ and $s$ are continuous we have $f(s(u.b_j))\To_j 0$; in particular
\[
\tilde{f}(b_j)(u,\bar{X}_1,\dots, \bar{X}_n)= f(s(u.b))(1,\bar{X}_1,\dots, \bar{X}_n) \underset{j}{\To} 0
\]
and thus $\tilde{f}(b_j) \to_j 0$, as desired. Lastly we have to provide the isomorphism between the two complexes. For this we need the Hopf algebra structure on $\Ug$ (see eg.~\cite{KS}). Denote the coproduct and antipode by $\Delta$ and $S$, respectively, and note that by the universal property of the topology on $\Ug$ both these maps are continuous (here the algebraic tensor product $\Ug\tens \Ug$ is also endowed with  the maximal locally convex topology). 
We will also be using the so-called leg numbering notation and write $\Delta(u)=u_{(1)}\tens u_{(2)}$ whenever convenient.  For each $f\in \Hom(\Ug\tens \wedge^n \mfg/\mfh, E)$, denote by $A_f\in \Hom(\Ug\tens \Ug\tens\wedge^n \mfg/\mfh, E)$ the operator
\[
A_f(u, v, \bar{X}_1,\dots, \bar{X}_n ):=uf(S(v),\bar{X}_1,\dots, \bar{X}_n).
\]
Note that since $E$ is locally convex and $A_f$ is linear by construction it is automatically con\-ti\-nuous; i.e.~in the $\Hom$-set as implicitly claimed above. Moreover, since $E$ is a topological $\mfg$-module, it follows  that the map 
\begin{align*}
A\colon \Hom(\Ug\tens\wedge^n \mfg/\mfh, E) &\To \Hom(\Ug\tens \Ug\tens\wedge^n \mfg/\mfh, E)\\
 f &\longmapsto A_f
\end{align*}
 is continuous.
Now define $B\colon \tilde{E}^n \to E^n$  by
\[
B(f)\left(u,\bar{X}_1,\dots, \bar{X}_n\right):= A_f\left(\Delta(u),\bar{X}_1, \dots, \bar{X}_n\right)=u_{(1)} f\left(S(u_{(2)}),\bar{X}_1,\dots, \bar{X}_n\right).
\]
Note that $B(f) = A_f\circ  (\Delta\tens \id_{\wedge^n \mfg/\mfh})$ is linear and hence  continuous as implicitly claimed above. 
One now needs to verify that $B$ is well-defined (i.e.~that it actually takes values in $E^n$) and that it is $\mfg$-linear and commutes with the coboundary maps. 
Similarly, defining $B'\colon E^n \to \tilde{E}^n$ by the same formula, one may check that $B'\circ B=\id_{\tilde{E}^n}$ and $B\circ B'=\id_{E^n}$. 
These claims and their proofs are of a purely algebraic nature, so we skip the details and refer to the proof of \cite[Lemme 2.6]{guichardet-book} for more details.
Lastly, we note that since $B(f) = A_f\circ (\Delta\tens \id_{\wedge^n \mfg/\mfh})$ the continuity of $A$ implies the continuity of $B$, and hence of $B'$ which is defined by the same formula.

\end{proof}
\begin{rem}\label{reductive-rem}
In the beginning of the proof of Proposition \ref{the-two-isomorphic-resolutions}, we skipped some of the algebraic details regarding the homological complex predual to \eqref{E-tilde-reso} and, in particular, the proof that it is exact and algebraically $\mfh$-strengthened. We note here that this is not completely trivial and is where our standing assumption that $\mfh$ be reductive is used.
\end{rem}

All of the above combines with the standard results of Appendix \ref{rel-hom-alg-appendix} to show that the following definition makes sense; i.e.~that the cohomology groups defined are unique up to linear homeomorphism.
\begin{defi}
Let $E$ be a locally convex topological $\mfg$-module. Then the (topological)  relative cohomology $\Cohom^n(\mfg, \mfh, E)$ is defined as the cohomology of the complex
\[
0\To \left(E^0\right)^{\mfg} \overset{d^0\restriction}{\To} \left(E^1\right)^{\mfg} \overset{d^1\restriction}{\To} \left(E^2\right)^{\mfg} \overset{d^2\restriction}{\To}\left(E^3\right)^{\mfg}\overset{d^3\restriction}{\To} \cdots
\]
for some/any $\mfh$-relative injective, $\mfh$-strengthened resolution 
\[
0 \To E \To  E^0 \overset{d^0}{\To} E^1\overset{d^1}{\To} E^2 \overset{d^2}{\To} E^3\overset{d^3}{\To} \cdots
\]
of $E$ in the category of topological $\mfg$-modules. Similarly, the  reduced cohomology $\underline{\Cohom}^n(\mfg, \mfh, E)$ is defined as the quotient of the closed subspace $\ker(d^n\restriction)$ by the closure of $\im(d^{n-1}\restriction)$.  When $\mfh=\{0\}$, the cohomology is denoted $\Cohom^n(\mfg,E)$ and $\underline{\Cohom}^n(\mfg,E)$, respectively.
\end{defi}

\begin{rem}
Note that the cohomology groups $\Cohom^n(\mfg, \mfh, E)$ are, as linear spaces,  nothing but the usual cohomology 
defined by forgetting that $E$ is a \emph{topological} $\g$-module and simply using $\mfh$-strengthened, $\mfh$-relative injective  resolutions in the category of \emph{algebraic} $\mfg$-modules. Thus, the novelty is only in the fact that the cohomology spaces carry a natural  (generally non-Hausdorff!), locally convex vector space topology. The fact that $\Cohom^n(\mfg,\mfh, E)$ may not be Hausdorff makes it unpleasant to work with from a functional analytic point of view, and the reduced cohomology is introduced exactly to remedy this problem:  $\underline{\Cohom}^n(\mfg, \mfh, E)$ is obtained from $\Cohom^n(\mfg, \mfh, E)$  by dividing out the closure of the class of zero, and is thus the biggest possible Hausdorff quotient of $\Cohom^n(\mfg, \mfh, E)$. This point of view also shows that  the topology on $\underline{\Cohom}^n(\mfg, \mfh, E)$ is canonically defined; i.e.~independent of the choice of injective resolution with which it is computed.
\end{rem}

\subsection{Computation by inhomogeneous cochains}\label{inhomogen-cochains-section}

In this section we show that $\Cohom^\bullet(\mfg, \mfh, E)$ can be computed, as a topological object, by the standard complex of inhomogeneous cochains. Consider the space
$\Hom_{\mfh}(\wedge^n \mfg/\mfh, E)$ consisting of those $f\in \Hom(\wedge^n \mfg/\mfh, E)$ which satisfy
\begin{align}\label{h-equiv-inhom}
Y.f(\bar{X}_1,\dots, \bar{X}_n)=\sum_{i=1}^n f(\bar{X}_1,\dots, \bar{X}_{i-1}, \overline{[Y,X_i]},\bar{X}_{i+1},\dots, \bar{X}_{n}),
\end{align}
endowed with the usual topology of uniform convergence on compacts, which, since $\wedge^n \mfg/\mfh$ is finite dimensional, is nothing but the topology of pointwise convergence. When $n=0$, this definition requires a bit of interpretation: the sum on the right hand side of \eqref{h-equiv-inhom} is empty, and hence zero, and  $\wedge^0\mfg/\mfh=\RR$ so that  $\Hom_{\mfh}(\wedge^0 \mfg/\mfh, E)$ identifies with $E^\mfh$. Define now a coboundary map $d^n\colon \Hom_{\mfh}(\wedge^n \mfg/\mfh, E) \to \Hom_{\mfh}(\wedge^{n+1} \mfg/\mfh, E)$ by setting
\begin{align}\label{inhomogen-coboundary-map}
d^n(f)\left(\bar{X}_1,\dots, \bar{X}_n\right) &:= \sum_{i=1}^{n+1} (-1)^{i+1} X_i.f\left(\bar{X}_1,\dots, \hat{\bar{X}}_i, \dots, \bar{X}_{n+1}\right) + \notag\\
 &+ \sum_{i<j} (-1)^{i+j} f\left(\overline{[X_i,X_j]}, \bar{X}_1,\dots, \hat{\bar{X}}_i,\dots, \hat{\bar{X}}_j,\dots, \bar{X}_{n+1}\right).
\end{align}
Denoting by by $(E^n, d^n)$ the standard resolution described in the previous section, we now have the following:
\begin{prop}\label{inhomogeneous-complex-prop}
The map $\alpha_n\colon (E^n)^{\mfg} \To \Hom_{\mfh}(\wedge^n\mfg/\mfh, E)$ defined by 
\[
\alpha_n(f)(\bar{X}_1,\dots, \bar{X}_n):=f(1, \bar{X}_1,\dots, \bar{X}_n)
\]
 is a topological isomorphism, with inverse $\alpha_n^{-1}(g)(u, \bar{X}_1,\dots, \bar{X}_n )= ug(\bar{X}_1,\dots, \bar{X}_n)$, and the collection $(\alpha_n)_{n}$ defines an isomorphism of complexes 
 \[
 \big((E^n)^\mfg, d^n \restriction \big) \simeq \big(\Hom_{\mfh}(\wedge^n\mfg/\mfh, E), d^n\big).
 \]
Hence the complex 
\begin{align*}
0\to E^{\mfh}\overset{d^0}{\To} \Hom_{\mfh}(\mfg/\mfh, E) \overset{d^1}{\To} \Hom_{\mfh}(\mfg/\mfh \wedge \mfg/\mfh, E) \overset{d^2}{\To} \Hom_{\mfh}(\wedge^3\mfg/\mfh, E) \overset{d^3}{\To} \cdots,
\end{align*}
known as the inhomogeneous cochain complex,  computes the topological Lie algebra cohomology $\Cohom^n(\mfg,\mfh,E)$.
\end{prop}
\begin{proof}
It is straight forward to see that the two maps are well defined and each others inverse and since the topology on both spaces is the topology of pointwise convergence, it is also clear that they are both continuous. A direct computation now verifies that the $\alpha_n$'s intertwine the coboundary maps and the proof is complete.
\end{proof}
\begin{rem}
As in most other cohomology theories, the degree one case is of special importance, so we digress for a moment to write down explicit formulas for the inhomogeneous cocycles and coboundaries: An inhomogeneous 1-cocycle is an $f\in \Hom_{\mfh}(\mfg/\mfh, E) $ which satisfies
\[
X.f(\bar{Y})- Y.f(\bar{X})-f(\overline{[X,Y]})=0,
\]
and $f$  is an inhomogeneous 1-coboundary exactly if there exists $\xi\in E^{\mfh}$ such that $f(\bar{X})=X.\xi$.
\end{rem}

\subsection{The Shapiro lemma}
In this section we prove that the Shapiro lemma (see e.g.~\cite[\S 7]{guichardet-book}) extends to the topological context. To this end, consider a Lie algebra $\mfg$ with a Lie sub-algebra $\mfh\leq \mfg$ (now no longer required to be reductive) as well as a locally convex topological $\mfh$-module $E$, and turn the space $\Hom_{\Uh}(\Ug, E)$ into a topological $\mfg$-module by setting
$(X.f)(u):=f(uX)$. When endowed with this action,  $\Hom_{\Uh}(\Ug, E)$ is called the \emph{coinduced module} and denoted $\Coind(E)$. In our context, the Shapiro lemma now states the following:
\begin{thm}[Topological Shapiro lemma]
Let $\mfh\leq \mfg$ be a  Lie sub-algebra and let $E$ be a topological $\mfh$-module. Then for all $n\in \NN_0$ there exists a topological isomorphism
\[
\Cohom^n(\mfh, E)\simeq \Cohom^n(\mfg, \Coind(E) ).
\]
\end{thm}

\begin{proof}
Consider first the twisted standard resolution of $\Coind(E)$; i.e.~the complex
\begin{align}\label{coind-reso}
&0\To \Coind(E) \To  \Hom\left(\Ug,\Coind(E)\right) \To \Hom\left(\Ug\tens \mfg, \Coind(E)\right) \To  \cdots
\end{align}
with $\mfg$-action and coboundary maps defined as in \eqref{tilde-g-action} and \eqref{tilde-coboundary-map}. Consider also the complex of topological $\mfh$-modules
\begin{align}\label{shapiro-h-reso}
&0\To E \To  \Hom(\Ug,E) \To \Hom(\Ug\tens \mfg, E) \To  \Hom(\Ug\tens \mfg\wedge \mfg, E)\To  \cdots
\end{align}
where $\mfh$ acts on $\Hom(\Ug\tens \wedge^p \mfg, E)$ as in the formula \eqref{tilde-g-action} and the coboundary maps are defined by \eqref{tilde-coboundary-map}. 
\begin{claim*}
The complex is \eqref{shapiro-h-reso} a strengthened, relative injective resolution of the $\mfh$-module $E$.
\end{claim*}
\begin{proof}[Proof of Claim.] From Proposition \ref{the-two-isomorphic-resolutions}, and its proof, we already know that the complex \eqref{shapiro-h-reso} is exact and strengthened, 
so what remains to be shown is that $\Hom(\Ug\tens \wedge^p \mfg, E)$ is a relative injective $\mfh$-module. For this, choose a linear basis $Y_1,\dots, Y_p$ for $\mfh$ and extend it by vectors $X_1,\dots, X_q$ to a linear basis for $\mfg$. By the Poincaré-Birkoff-Witt theorem, the vectors $X_1^{s_1}\cdots X_{q}^{s_q} Y_1^{r_1}\dots Y_p^{r_p}$ where $r_i,s_i\in \NN_0$ 
is a linear basis for $\Ug$, and for each multi-index $\mathbf{t}=(t_1,\dots, t_q)\in \NN_0^q$ we define $p_{\mathbf{t}}\colon \Ug\to \Uh$ by setting
\[
p_\mathbf{t}(X_1^{s_1}\cdots X_{q}^{s_q} Y_1^{r_1}\dots Y_p^{r_p}):= \begin{cases}Y_1^{r_1}\dots Y_p^{r_p}& \mbox{if } \mathbf{t}=(s_1,\dots, s_q) \\
0 & \mbox{otherwise}  \end{cases} 
\]
Defining $X^{\mathbf{t}}:=X_1^{t_1}\dots X_q^{t_q}$, every $u\in \Ug$ can be written as $u=\sum_{\mathbf{t}\in \NN_0^q}  X^{\mathbf{t}} p_{\mathbf{t}}(u)$, 
with only finitely many non-zero summands, and each of the maps $p_{\mathbf{t}}$ is left $\Uh$-linear by construction and continuous since $\Uh$ is locally convex.
Consider now the setup of relative injectivity; i.e.~two continuous $\mfh$-modules $A,B$, a strengthened injective $\mfh$-morphism $\iota\colon A\to B$ and a $\mfh$-morphism $f\colon A\to \Hom(\Ug\tens \wedge^n\mfg,E)$. Choose a continuous left inverse $s\colon B\to A$ of $\iota$ and define $\tilde{f}\colon B\to \Hom(\Ug\tens \wedge^n\mfg,E)$ by setting
\[
\tilde{f}(u,Z_1,\dots, Z_n):=\sum_{\mathbf{t}\in \NN_0^q} f\left(s(p_{\mathbf{t}}(u).b)\right)(X^{\mathbf{t}}, Z_1,\dots, Z_n).
\]
Then a direct computation shows that $\tilde{f}\circ \iota=f$ and that $\tilde{f}$ is $\mfh$-linear, and since $f, p_{\mathbf{t}}$ and $s$ are continuous so is $\tilde{f}$. 
\end{proof}
Passing to $\mfg$- and $\mfh$-invariants in \eqref{coind-reso} and \eqref{shapiro-h-reso}, respectively, 
we therefore have that the induced complexes
\begin{align*}
&0 \To  \Hom(\Ug,\Coind(E))^\mfg \To \Hom(\Ug\tens \mfg, \Coind(E))^\mfg \To  \cdots\\
&0\To  \Hom(\Ug,E)^\mfh \To \Hom(\Ug\tens \mfg, E)^\mfh \To  \cdots
\end{align*}
compute $\Cohom^\bullet(\mfg, \Coind(E))$ and $\Cohom^\bullet(\mfh, E)$, respectively. Fix an $n\in \NN$ and consider the map
\[
\alpha_n\colon \Hom(\Ug\tens \wedge^n \mfg, \Coind(E))^\mfg \To \Hom(\Ug \tens \wedge^n\mfg, E)^\mfh
\]
given by $\alpha_n(f)(u,X_1,\dots, X_n):=f(u,X_1,\dots, X_n)(1)$. A direct computation shows that $\alpha_n$ is indeed well defined, i.e.~takes values in $\Hom(\Ug \tens \wedge^n\mfg, E)^\mfh$, and that the map
\[
\beta_n\colon   \Hom(\Ug \tens \wedge^n\mfg, E)^\mfh \To  \Hom(\Ug\tens \wedge^n \mfg, \Coind(E))^\mfg
\]
given by $\beta_n (f) (u,X_1,\dots, X_n) (v):= f(vu,X_1,\dots, X_n)$ is inverse to $\alpha_n$.  
Furthermore, it is straight forward to see that both $\alpha_n$ and $\beta_n$ are continuous maps which commute with the coboundary maps, and hence they induce mutually  inverse topological isomorphisms  between $\Cohom^n(\mfg, \Coind(E))$ and $\Cohom^n(\mfh, E)$.
\end{proof}


\section{The van Est theorem}\label{top-van-est-thm-section}
In this section we connect the cohomology of a Lie group  with the cohomology of its Lie algebra. We will assume familiarity with the basics on cohomology for locally compact groups, and everything needed can be found in \cite{guichardet-book} or \cite{borel-wallach}. We briefly recall, that this theory is developed along the same lines as the cohomology theory for Lie algebras presented in Appendix \ref{rel-hom-alg-appendix}, with the objects now being topological modules for a locally compact group $G$; more precisely, one first checks that every topological $G$-module $E$ has a relative injective, strengthened resolution (within the category of topological $G$-modules) and defines the cohomology $\Cohom^n(G,E)$ as the cohomology of the complex obtained by passing to $G$-invariants in such a strengthened injective resolution. The arguments in Appendix \ref{rel-hom-alg-appendix}, which show that the Lie algebra cohomology is well defined as a topological object, can be applied, mutatis mutandis, to show that $\Cohom^n(G,E)$ is well defined as a (generally non-Hausdorff) topological vector space and one then proceeds to define the reduced cohomology $\underline{\Cohom}^n(G,E)$ as $\Cohom^n(G,E)$ modulo the closure of the zero class. When the $G$-module $E$ in question has pleasant features (e.g.~quasi completeness, local convexity, completeness, Fr{\'e}chet etc.) the standard resolution (see \cite[Chapter III]{guichardet-book}) inherits these properties and one may then restrict attention to the class of modules having this property. More precisely, if $E$ is a, say, complete topological $G$-module then the cohomology $\Cohom^n(G,E)$ can be computed using any strengthened resolution consisting of complete topological $G$-modules which are relative injective \emph{within} the category of complete topological $G$-modules.  The importance of being able to pass to a smaller class of  $G$-modules lies in the fact that the extra structure possessed by these modules often makes it easier to prove relative injectivity  --- for a concrete instance of this phenomenon, see the proof of relative injectivity of the space of vector valued differential forms below. \\

Let now $G$ be a connected Lie group with Lie algebra  $\mfg$ (see e.g.\cite{warner} for an introduction to these notions) 
and let $E$ be a complete, locally convex $G$-module. Recall that a vector $\xi \in E$ is called \emph{smooth}  if the function $G\ni g\mapsto g.\xi\in E$ is smooth and that
\[
X.\xi:= \lim_{t\to 0}\frac{ \exp(tX).\xi-\xi }{t}, \ X\in \mfg,
\]
defines an action of the Lie algebra $\mfg$ on the space $E^\infty$ of smooth vectors. Furthermore,  $X\in \mfg$  acts as a continuous operator   when $E^\infty$ is endowed with the smooth topology arising from the embedding $j\colon E^\infty \to C^\infty(G,E)$ (see Appendix \ref{smooth-top-appendix}  for details), and hence $E^\infty$ is a topological $\mfg$-module \cite[Section 4.4]{warner-harmonic-analysis-I}.  The $G$-module $E$ is said to be smooth if every $\xi \in E$ is smooth and the mapping $j\colon E \to C^\infty(M,E)$ is a homeomorphism onto its image. Note that the map $j$ is always continuous, so if $E$ is Fr{\'e}chet then the open mapping theorem implies that $j$ is automatically a homeomorphism once $E=E^\infty$.
In the following,  we will restrict attention to the case where $E$ is Fr{\'e}chet  in order to have the open mapping theorem at our disposal. To prove the van Est theorem below we will need the following lemma; which is actually nothing but the statement of the theorem in degree 0.

\begin{lem}\label{G-fix-lig-g-fix}
If $G$ is a connected Lie group and $E$ is a smooth Frechet $G$-module then $E^G=E^\mfg$.
\end{lem}
\begin{proof}
If $\xi\in E^G$ then clearly $X.\xi=0$ for all $X\in \mfg$,  so $\xi\in E^\mfg$. Conversely, assume that $\xi\in E^{\mfg}$  and let $\eta\in E^*$ (the dual of $E$)  be given and consider the smooth function $f(g):=\eta(g.\xi)$. Choose a basis $X_1,\dots, X_n$ for $\mfg$ and an open zero-neighbourhood $V\subset \RR^n$ on which the map
$
(t_1,\dots, t_n) \mapsto \exp\left(\sum_i t_iX_i\right)
$
is (the inverse of) a coordinate system \cite[Theorem 3.31]{warner}.  
For a given $g_0\in G$, the function
$
V\ni (t_1,\dots, t_n) {\longmapsto} g_0\exp\left(\sum_i t_iX_i \right) \in G
$
defines coordinates $(x_1,\dots, x_n)$ around $g_0$ and computing the derivatives of $f$ in these coordinates we get:
\begin{align*}
\partafl{f}{x_i}\Big{\rvert}_{g_0} &= \lim_{h\to 0 }\frac{f\left(g_0\exp(hX_i)\right)- f(g_0)}{h} =\lim_{h\to 0}\frac{\eta\left(g_0.(\exp(hX_i).\xi- \xi)\right)}{h} =\\
&=(g_0^{-1}.\eta)\left(\lim_{h\to 0}\frac{\exp(hX_i).\xi-\xi}{h} \right)=(g_0^{-1}.\eta)(X_i.\xi)=0.
\end{align*}
Thus, all derivatives of $f$ vanish at all points and since $G$ is assumed connected, this means   that $f$ is constant \cite[Theorem 1.24]{warner}; that is,  $\eta(g.\xi)=\eta(\xi)$ for all $g\in G$. But since this holds for all $\eta\in E^*$ and $E^*$ separates points in $E$, we conclude that $g\xi=\xi$ for all $g\in G$, and hence that $\xi\in E^G$ as desired.
\end{proof}

\begin{thm}[Topological van Est theorem]\label{top-van-est-thm} If $G$ is a connected Lie group with Lie algebra $\mfg$, $K\leq G$ is a maximal compact subgroup with Lie algebra $\mfk$ and $E$ is a smooth Fr{\'e}chet $G$-module, then there exists  a bicontinuous linear isomorphism $\Cohom^n(G,E) \simeq \Cohom^n(\mfg, \mfk, E)$ for each $n\in \NN_0$.
\end{thm}


As indicated multiple times already, this theorem goes back to van Est \cite{van-est} (see also \cite{hochschild-mostow, mostow}). As a precursor to the topological part of statement in Theorem \ref{top-van-est-thm}, we remark the following  which was noted already in \cite{hochschild-mostow}:  when $E$ is finite dimensional, then clearly $\Cohom^n(\mfg,\mfk, E)$ is finite dimensional as well, and hence, by the algebraic van Est theorem,  so is $\Cohom^n(G,E)$. In this case, by \cite[III, Proposition 2.4]{guichardet-book} the cohomology $\Cohom^n(G,E)$ is automatically Hausdorff and hence both sides of the van Est isomorphism are finite dimensional, locally convex, Hausdorff topological vector spaces and thus \emph{any} linear isomorphism is automatically a homeomorphism. To prove the general statement in Theorem \ref{top-van-est-thm}  we follow the strategy of proof presented in \cite{guichardet-book}, keeping track of the topologies.\\
\begin{proof}
Since $K$ is closed, the quotient space $M:=G/K$ is again a manifold upon which $G$ acts by diffeomorphisms via the maps induced by left translation.
That is, for each $g\in G$ we get a diffeomorphism $l_g\colon M\to M$ given by $l_g(\bar{x})=\overline{gx}=g.\bar{x}$ (where $\bar{x}$ denotes the class in $M$ of $x\in G$) and we therefore get induced isomorphisms $\lambda_g\restriction_m:= dl_g\restriction_m\colon T_mM \to T_{g.m}M$. Note that $\lambda_{k}\restriction_{\bar{e}}$ maps $T_{\bar{e}}M$ to $T_{\bar{e}}M$ whenever $k\in K$ and one easily shows that this defines an action $\lambda$ of $K$ on $T_{\bar{e}}M$. 
The manifold structure on $M$ is defined such that the natural quotient map $p\colon G\to M$ is a submersion (see eg.~\cite[Theorem 3.58]{warner}) and hence the induced map $\pi:=dp\colon \mfg \to T_{\bar{e}}M$ is a surjection. Furthermore, one sees that $\pi(X)=0$ whenever $X\in \mfk$ and hence the induced map $\bar{\pi}\colon \mfg/\mfk \to T_{\bar{e}}M$ is a linear isomorphism by dimension count.
\begin{claim}
The map ${\pi}$ intertwines the $\Ad$-action of $K$ on $\mfg$ with the $\lambda$-action of $K$ on $T_{\bar{e}}M$. 
\end{claim}
\begin{proof}[Proof of Claim 1] Denote by $a$  the conjugation action $a_g(x):=gxg^{-1}$ of $G$ on itself and by $\Ad$ the induced action on $\mfg$ given by
$
\Ad_g:=da_g\restriction_{T_eG}\colon \mfg \to \mfg.  
$ 
Then $a\circ \exp= \exp\circ \Ad$ \cite[3.46]{warner}, and for $k\in K$ and $f\in C^\infty(M)$ we therefore have
\begin{align*}
\pi\left(\Ad_k(X)\right)(f)&= \Ad_k(X)(f\circ p) = \frac{d}{dt}\Big{\rvert}_{t=0} f\circ p\left(\exp(t \Ad_k(X))\right)=\\
& = \frac{d}{dt}\Big{\rvert}_{t=0} (f\circ p)(a_k \circ \exp(tX))=\lim_{t\to 0} \frac{f\left(k.\overline{\exp(tX)k^{-1}}\right)- f(\bar{e})}{t}=\\
&=\lim_{t\to 0} \frac{f\left(k.\overline{\exp(tX)}\right)- f(\bar{e})}{t}=\lambda_k(\pi(X))(f).
\end{align*}
\end{proof}
Let $G$ act on $C^\infty(M)$ by left translation, on the smooth vector fields  $\Vect(M)$ as $(g.X)(f):= g(X(g^{-1}f))$ and on  the space $\Omega^p(M,E)$ of $E$-valued $p$-forms (see Appendix \ref{smooth-top-appendix} for definitions) by
\[
(g.\omega)(X_1,\dots, X_n):= g.\omega(g^{-1}X_1,\dots, g^{-1}X_p), \quad X_1,\dots, X_p\in \Vect(M).
\]
Localized in a point $m\in M$, the latter action is given by
\[
(g.\omega)_m:=g.(\omega_{g^{-1} m} \circ \wedge^p \lambda_{g^{-1}}\restriction_m (-)).
\]
We endow $\Omega^p(M,E)$ with the initial topology arising from the maps
\[
\Omega^p(M,E)\ni \omega \longmapsto \omega(X_1,\dots, X_p)\in C^\infty(M,E); \quad X_1,\dots, X_p\in \Vect(M)
\]
Since $C^\infty(M,E)$ is complete (in the smooth topology), $\Omega^p(M,E)$ is complete and locally convex in this topology and the $G$-action is continuous (one may actually prove that the action is smooth, but we shall not need this fact so we omit the details). 
\begin{claim}
The $G$-module $\Omega^p(M,E)$ is relative injective in the category of complete topological $G$-modules.
\end{claim}
\begin{proof}[Proof of Claim 2.]
Consider the  usual setup for relative injectivity:
\[
\xymatrix{
0 \ar[r] & A \ar[d]_f \ar[r]_\iota & B  \ar@/_/@{->}[l]_{s} \ar@{-->}[dl]^{\exists? \tilde{f}}  \\
 & \Omega^p(M,E) &
}
\]
where $s\colon B\to A$ is continuous with $s\circ \iota=\id_A$. Fix a $b\in B$. The map $\gamma\colon K\to A$ given by $\gamma(k)=k.s(k^{-1}.b)$ is continuous and since $A$ is assumed complete (actually, quasi-completeness suffices) we may therefore integrate $\gamma$ as a function into $A$ \cite[VI6, Proposition 8]{bourbaki-integration} and define $\tilde{s}\colon B\to A$ by
\[
\tilde{s}(b):=\int_{k\in K} k.s(k^{-1}.b){d}k, 
\]
where we integrate against the Haar measure on $K$. Then $\tilde{s}$ is $K$-equivariant, continuous and still satisfies $\tilde{s}\circ \iota=\id_A$, and the map $\tilde{f}\colon B\to \Omega^p(M,E)$ is now defined as 
\[
\tilde{f}(b)\restriction_{\bar{g}}:= f\left(g.(\tilde{s}(g^{-1}.b))\right)\restriction_{\bar{g}}
\]
The $K$-equivariance of $\tilde{s}$ makes $\tilde{f}$ well-defined (i.e.~independent of the choice of representative $g$ for the class $\bar{g}$) and a straight forward calculation
shows that $\tilde{f}$ is $G$-linear and satisfies $\tilde{f}\circ \iota=f$. \end{proof}

We now define a complex
\begin{align}\label{omega-reso}
0\To E \overset{\varps}{\To} \Omega^0(M,E) \overset{\delta^0}{\To} \Omega^1(M,E) \overset{\delta^1}{\To} \Omega^2(M,E) \overset{\delta^2}{\To} \cdots
\end{align}
where $\varps(\xi)(m):=\xi$ and 
\begin{align}\label{exterior-derivative-def}
\delta^p(\omega)(X_1,\dots, X_{p+1})&:= \sum_{i=1}^{p+1}(-1)^{i+1} X_i.\omega(X_1,\dots, \hat{X_i}, \dots, X_{p+1}) + \notag \\
&+\sum_{i<j}(-1)^{i+j} \omega\left([X_i,X_j],X_1,\dots, \hat{X}_i,\dots, \hat{X}_j,\dots, X_{p+1} \right).
\end{align}
One may check that \eqref{omega-reso} is a complex and it is easy to see that the coboundary maps are $G$-equivariant and continuous, so that \eqref{omega-reso} is indeed a complex of continuous $G$-modules. Since $G$ is assumed connected and $K$ is a maximal compact subgroup, by the Malcev-Iwasawa theorem \cite[Theorem 5.11]{osborn-vector-bundles} there exists a $d\in \NN$ such that $G$ is diffeomorphic with $\RR^d\times K$, and hence $M$ is diffeomorphic with $\RR^d$. Since $E$ is assumed Fr{\'e}chet, the Poincaré lemma \cite[4.18]{warner} now generalizes to the case of vector valued forms and provides a continuous contracting homotopy of \eqref{omega-reso}. This shows that \eqref{omega-reso} is strengthened (see e.g.~Proposition \ref{strongly-exact-gives-contractible} and Remark \ref{converse-rem}), relative injective resolution of $E$ and hence $\Cohom^\bullet(G,E)$ can be computed using the complex
\begin{align*}
0\To \Omega^0(M,E)^G \overset{\delta^0\restriction}{\To} \Omega^1(M,E)^G \overset{\delta^1\restriction}{\To} \Omega^2(M,E)^G \overset{\delta^2\restriction}{\To} \cdots
\end{align*}
We now need to relate the cohomology of this complex to the cohomology of the Lie algebra $\mfg$. To this end, define a map
$
\alpha \colon \Omega^p(M,E)^G \To \Hom(\wedge^p T_{\bar{e}}, E)
$
by setting $\alpha(\omega)(X_1,\dots, X_p):=\omega_{\bar{e}}(X_1,\dots, X_p)$. On $\Hom(\wedge^p T_{\bar{e}}, E)$ we have a natural action of $K$ given by
\[
(k.\varphi)(X_1,\dots, X_p):=k.\varphi(\lambda_{k^{-1}}X_1,\dots, \lambda_{k^{-1}}X_p),
\]
and $\alpha$ takes values in the $K$-invariant elements, as is seen from the following computation:
\begin{align*}
(k.\alpha(\omega))(X_1,\dots, X_p) & := k\left(\omega_{\bar{e}}(\lambda_{k^{-1}}X_1,\dots, \lambda_{k^{-1}}X_p)\right)=(k.\omega)_{\bar{k}}(X_1,\dots, X_p)=\\
& =\omega_{\bar{e}}(X_1,\dots, X_p)=\alpha(\omega)(X_1,\dots, X_p).
\end{align*}
Now define a map $\beta\colon \Hom(\wedge^pT_{\bar{e}}, E)^K\to \Omega^p(M,E)^G$ by setting
\[
\beta(\varphi)_{\bar{g}}:= g.\left(\varphi \circ \wedge^p\lambda_{g^{-1}}\restriction_{T_{\bar{g}}M} (-)  \right)
\] 
(this is well-defined exactly because $\varphi$ is fixed by $K$). A direct computation now shows that $\alpha$ and $\beta$ are each others inverses, and we now prove that they are both continuous. Recall that a sequence $\omega_i\in \Omega^p(M,E)$ converges to zero iff $\omega_i(X_1,\dots, X_p)\to 0$ in $C^\infty(M,E)$ for all $X_1,\dots, X_p\in \Vect(M)$. If this is the case, then clearly
\[
\alpha(\omega_i)(X_1\restriction_{\bar{e}}, \dots, X_p\restriction_{\bar{e}})=\omega_i(X_1,\dots, X_p)(\bar{e})\underset{i\to \infty}{\To} 0;
\]
that is, $\alpha(\omega_i)$ converges pointwise to zero and hence the map $\alpha$ is continuous. Since $E$ assumed Fr{\'e}chet, both the domain and range of $\alpha$ is Fr{\'e}chet as well, and hence, by the open mapping theorem, $\beta$ is also continuous. 
Next, consider the space $\Hom(\wedge^p\mfg/\mfk, E)$ and endow it with the $K$-action defined by
\[
(k.\varphi)(\bar{X}_1,\dots, \bar{X}_{p}):= k.\varphi\left( \overline{\Ad_{k^{-1}}X_1}, \dots, \overline{\Ad_{k^{-1}}X_p}  \right)
\]
By Claim 1 above, precomposition with the map $\bar{\pi}\colon \mfg/\mfk\to T_{\bar{e}}M$ yields a (topological) isomorphism of $K$-modules $\gamma\colon \Hom(\wedge^p\mfg/\mfk, E)\To \Hom(\wedge^pT_{\bar{e}}, E)$ and hence we get an isomorphism
\[
\gamma^{-1}\circ \alpha \colon \Omega^p(M,E)^G \To \Hom(\wedge^p \mfg/\mfh, E)^K.
\]
Since the $\Ad$-action of $K$ on $\mfg$ differentiates to the $\ad$-action (given by $\ad_Y(X):=[Y,X]$)  of $\mfk$ on $\mfg$  \cite[Proposition 3.47]{warner}, a direct computation verifies that the $K$-module $\Hom(\wedge^p \mfg/\mfh, E)$ is smooth and that its $K$-action differentiates to the action given by
\[
(Y.\varphi)(\bar{X}_1,\dots, \bar{X}_p):=Y.\varphi(\bar{X}_1, \dots, \bar{X}_p) - \sum_{i=1}^p\varphi(\bar{X}_1,\dots, \overline{[Y,X_i]},\dots, \bar{X}_p), \quad Y\in \mfk;
\]
where $Y$ acts on the first summand by means of the differentiated action on $E$. Since $G$ is assumed connected and is diffeomorphic with $\RR^d\times K$,   $K$ is also connected and by  Lemma \ref{G-fix-lig-g-fix} the $K$-invariants and $\mfk$-invariants therefore agree;  we thus have that $\Hom(\wedge^p\mfg/\mfk, E)^\mfk$ is exactly the space of inhomogeneous $p$-cochains defined in Section \ref{inhomogen-cochains-section}. Chasing through the isomorphisms shows that $\gamma^{-1}\circ \alpha$ intertwines the coboundary map \eqref{exterior-derivative-def} with the inhomogeneous coboundary map \eqref{inhomogen-coboundary-map},  and the proof is complete.
\end{proof}


\begin{rem}
The first cohomology group $\Cohom^1(G, E)$, as well as its reduced ditto $\underline{\Cohom}^1(G,E)$, is of special importance in many applications, and we note here that the topological van Est theorem implies that if $E$ is a smooth Fr{\'e}chet $G$-module and $\underline{\Cohom}^1(G,E)=\{0\}$, then  also $\underline{\Cohom}^1(\mfg,\mfk, E)=\{0\}$. Hence, for any inhomogeneous $1$-cocycle $c\colon \mfg/\mfk \to E $ there exists a sequence $(\xi_i)_i$ in $E$ such that $c(\bar{X})=\lim_{i} X.\xi_i$. 
\end{rem}

\begin{rem}
By \cite[Theorem 5.2]{blanc}, whenever $G$ is a (connected) Lie group and $E$ is a continuous Fr{\'e}chet $G$-module, the inclusion $E^\infty \To E$ induces a topological isomorphism $\Cohom^n(G,E^\infty) \To \Cohom^n(G,E)$, so if $\underline{\Cohom}^n(G,E)$ vanishes then so does $\underline{\Cohom}^n(\mfg,\mfk, E^\infty)$. 
If $G$ is amenable and unimodular this is the case for all $n\geq 1$ when $E=L^2(G)$ \cite[Theorem C \& Porism 2.10]{KPV}, and in the special case $n=1$
it even holds for  any infinite-dimensional, irreducible, unitary Hilbert $G$-module $E$  (see \cite[Theorem 3.1]{florian-martin-reduced-1-cohomology}).

\end{rem}

\appendix

\section{Relative homological algebra}\label{rel-hom-alg-appendix}
In this appendix we work through the necessary relative homological algebra to justify that the relative Lie algebra cohomology is well-defined as a topological object. 
The exposition is in complete analogy with the theory developed for groups in \cite[Chapter 3]{guichardet-book} and is primarily included for the readers convenience.
In what follows,   $\mfg$ denotes a Lie algebra and $\mfh\leq \mfg$  a reductive Lie sub-algebra. 
We shall freely use the terminology introduced in Section \ref{top-lie-alg-comhom-section}; in particular the notion of $\mfh$-relative injective $\mfg$-modules and $\mfh$-strengthened maps will play a prominent role.
\begin{prop}\label{strongly-exact-gives-contractible}
If $0 \to E\overset{\varps}{\To}E^0\overset{d^0}{\To} E^1 \overset{d^1}{\To} \cdots$ is a $\mfh$-strengthened resolution of $E$ then it admits a contractive homotopy $(s^n)_{n\in \NN_0}$ consisting of continuous,  $\mfh$-linear maps. That is,  $s^n\colon E^n\to E^{n-1}$ is continuous and $\mfh$-linear (here $E^{-1}:=E$), 
\begin{align}\label{nul-homotopi}
 s^0\circ\varps=\id_E  \ \text{ and }  \ s^{l+1}\circ d^l +d^{l-1}\circ s^l=\id_{E^l} \text{ for all } l\in \NN.
\end{align}
\end{prop}
\begin{proof}
The fact that  $s^0$ exists follows directly from the assumption that $\varps$ is strengthened, so let us now treat the general case.
By assumption, the inclusion $\iota^l\colon\ker(d^l)\to E_l$ has a continuous $\mfh$-linear left inverse $\alpha^l\colon E^l\to\ker(d^l)$. 
Now consider the map $g^l\colon E^l\to E^l$ given by $g^l(x):=x-\alpha^l(x)$ and notice that $g^l$ vanishes on $\ker(d^l)$, so that we get an induced map $\bar{g}^l\colon E^l/\ker(d^l)\to E^l$.  Now choose a continuous $\mfh$-linear left inverse $\sigma^l\colon E^{l+1}\to E^l/\ker(d^l) $ to the induced map $\bar{d}^l\colon E^l/\ker(d^l)\to E^{l+1} $ and note that the restriction $\sigma_0^l:=\sigma^l\restriction_{\im(d^l)} $ is a left inverse to the corestriction $\bar{d}^l_0\colon E^l/\ker(d^l)\to \im(d^l)$. Note also that since $\overline{d}^l_0$ is bijective, $\sigma^l_0$ is actually a two-sided inverse which we will use to prove \eqref{beta-computation} below. Now define $ \beta^{l+1}:=\bar{g}^{l}\circ \sigma_0^l \colon \im(d^l)\to E_l$, and note that we have the following identity:
\begin{align}\label{beta-computation}
d^{l-1}\circ \beta^l(y) &=d^{l-1} \circ \bar{g}^{l-1}\circ \underbrace{\sigma_0^{l-1}(y)}_{=:\bar{z}}=d^{l-1} \circ g^{l-1}(z)=\notag\\
&=d^{l-1}(z-\alpha^{l-1}(z))=\bar{d}^{l-1}_0(\bar{z})= \bar{d}_0^{l-1}\circ \sigma_0^{l-1}(y)=y.
\end{align}
Defining $s^l:=\beta^l\circ \alpha^l$ now does the job:
\begin{align*}
(s^{l+1}\circ d^l+d^{l-1}\circ s^l)(x) &= \beta^{l+1}\circ \alpha^{l+1}\circ d^l(x) + d^{l-1}\circ \beta^l\circ \alpha^l(x) \overset{\eqref{beta-computation}}{=} \beta^{l+1}\circ d^l(x) +\alpha(x)=\\
&=\beta^{l+1}\circ\bar{d}^l_0 (\bar{x})+\alpha(x) =\bar{g}^l\circ \sigma_0^l \circ \bar{d}^l_0(\bar{x})+\alpha(x)=\bar{g}^l(\bar{x}) +\alpha(x)=\\
&=x-\alpha_l(x) +\alpha(x)=x.
\end{align*}

\end{proof}

\begin{rem}\label{converse-rem}
Actually, the converse to Proposition \ref{strongly-exact-gives-contractible} also holds true, in the sense that any resolution $0 \to E \overset{\varps}{\to} E^0 \overset{d^0}{\To}E^1\overset{d^1}{\To}\cdots$ of $E$ within the category of topological $\mfg$-modules, be it $\mfh$-relative injective or not, which admits a family of continuous, $\mfh$-linear maps $(s_n)_{n\in \NN_0}$ satisfying \eqref{nul-homotopi} is $\mfh$-strengthened.  We leave the proof as an exercise.
\end{rem}

\begin{prop}\label{prop-fundamental-extension-result}
Consider a topological $\mfg$-module $E$ and a $\mfh$-strengthened resolution 
\begin{align}\label{E-complex}
\xymatrix{
0\ar[r] & E \ar[r]^{\varps}  & E^0\ar[r]^{d^0} \ar@/^0.7pc/[l]^{s^0} & E^1 \ar[r]^{d^1} \ar@/^0.7pc/[l]^{s^1} & E^2\ar[r]^{d^2} \ar@/^0.7pc/[l]^{s^2} & \cdots \ar@/^0.7pc/[l]^{s^3}
}
\end{align}
where the maps $s^n$ are the ones ensured by Proposition \ref{strongly-exact-gives-contractible}. Let $F$ be another topological $\mfg$-module and consider a complex 
\begin{align}\label{F-complex}
\xymatrix{
0\ar[r] & F \ar[r]^{\eta}  & F^0\ar[r]^{\delta^0} & F^1 \ar[r]^{\delta^1}  & F^2\ar[r]^{\delta^2}  & \cdots 
}
\end{align}
in which every $F^n$ is $\mfh$-relative injective. Then for any $\mfg$-morphism $u\colon E\to F$ there exists a morphism of complexes $u^\bullet\colon (E^\bullet,d^\bullet)\To (F^\bullet, \delta^\bullet)$  consisting of continuous $\mfg$-linear maps which lifts\footnote{meaning that if the morphism $u^\bullet$ is augmented by $u$ in degree -1, then we obtain a morphism of complexes from \eqref{E-complex} to \eqref{F-complex}.} $u$.  Moreover, any morphism of complexes $u^\bullet$ lifting the zero-morphism from $E$ to $F$ is chain homotopic to the zero-morphism through a homotopy consisting of continuous $\mfg$-linear maps.
\end{prop}


\begin{proof}
We first prove the existence of  $u^\bullet$.  Since $F^0 $ is injective and $\varps$ is strengthened the map $\eta\circ u: E\to F^0$ extends to a $\mfg$-morphism $u^0\colon E^0\to F^0$. Assume, inductively, that $u^0, \dots, u^{n-1}$ has been constructed, and  consider the diagram
\[
\xymatrix{
E^{n-2} \ar[r]^{d^{n-2}} \ar[d]^{u^{n-2}} & E^{n-1}\ar[r]^{d^{n-1}} \ar[d]^{u^{n-1}} & E^n \ar[r]^{d^n} & E^{n+1}\\
F^{n-2} \ar[r]^{\delta^{n-2}} & F^{n-1}\ar[r]^{\delta^{n-1}} & F^n \ar[r]^{\delta^n} & F^{n+1}\\
}
\]
By commutativity of the diagram, the composition $\delta^{n-1}\circ u^{n-1}\circ d^{n-2}$ is zero and thus $\delta^{n-1}\circ u^{n-1}$ vanishes on $\im(d^{n-2})=\ker(d^{n-1})$. It therefore induces a map 
\[
\overline{\delta^{n-1}\circ u^{n-1}}\colon E^{n-1}/\ker(d^{n-1})\To F^n. 
\]
Since $d^{n-1}$ is strengthened so is the induced map $\bar{d}_0^{n-1}\colon E^{n-1}/\ker(d^{n-1})\To \im(d^{n-1})$, which means that this map is invertible as a continuous morphism of $\mfg$-modules. Denote by $\sigma$ its inverse and by $r\colon \im(d^{n-1})\to F^n$ the composition $\overline{\delta^{n-1}\circ u^{n-1}}\circ \sigma$.
Consider now the natural decomposition of $d^{n-1}$ as
\[
E^{n-1}\overset{d^{n-1}_0}{\To} \im(d^{n-1})\overset{j}{\hookrightarrow} E^n,
\]
and put $t:=d^{n-1}_0\circ s^n\colon E^n\to \im(d^{n-1})$. Then we have
\[
t\circ j\circ d^{n-1}_0=d^{n-1}_0\circ s^n\circ j\circ d^{n-1}_0=d^{n-1}_0\circ s^n \circ d^{n-1}=d^{n-1}_0\circ (\id_{E^{n-1}}- d^{n-2}\circ s^{n-1})= d^{n-1}_0.
\]
This shows that the inclusion map $j$ is $\mfh$-strengthened with left inverse $t$. Thus, by $\mfh$-relative injectivity of $F^n$, there exists a $\mfg$-morphism $u^n\colon E^n\to F^n$ such that $u^n\circ j=r$ and this now does the job:
\begin{align*}
u^n \circ d^{n-1}(x)&=u^n\circ j \circ d^{n-1}_0(x)=r \circ d^{n-1}_0(x)=\overline{\delta^{n-1} \circ u^{n-1}}\circ\sigma(d^{n-1}_0(x)) =\\
&=\overline{\delta^{n-1} \circ u^{n-1}}\circ\sigma(\bar{d}^{n-1}_0(\bar{x}))=\overline{\delta^{n-1} \circ u^{n-1}}(\bar{x})=\delta^{n-1} \circ u^{n-1}(x).
\end{align*}
For the second claim, assume that $u^\bullet$ is a lift of the zero map $0\colon E\to F$. We need to produce a sequence of maps $t^n\colon E^n\to F^{n-1}$ such that
\[
t^1\circ d^0=u^0 \ \text{ and } \  t^{n+1}\circ d^n +\delta^{n-1} \circ t^n=u^n.
\]
Since $u^{\bullet}$ lifts the zero map,  $u^0\circ \varps= \eta\circ 0=0$ so  $u^0$ induces a morphism  $\bar{u}^0\colon E^0/\ker(d^0)\to F^0$. Since $d^0$ is assumed strengthened, we get a continuous, $\mfh$-linear  inverse map $\sigma\colon \rg(d^0)\to E^0/\ker(d^0)$ and because $F^0$ is assumed $\mfh$-relative injective we get a $\mfg$-morphism $t^1\colon E^1\to F^0$ fitting into the diagram
\[
\xymatrix{
0\ar[r] & \rg(d^0) \ar[d]_\sigma \ar[r]^{\subset} & E^1 \ar@{-->}[ddl]^{t^1}\\
  &E/\ker(d^0) \ar[d]_{\bar{u}^0} & \\
& F^0 &
} 
\]
(the fact that the inclusion is strengthened follows again since $\rg(d^0)=\ker(d^1)$ and $d^1$ is assumed to be strengthened). By construction 
\[
t^1\circ d^0=\bar{u}^0\circ \sigma\circ d^0 =u^0,
\]
as desired. We now build the map $t^2$, and from this it will be clear how to proceed inductively. We want a $\mfg$-morphism $t^2\colon E^2\to F^1$ such that $t^2 \circ d^1 + \delta^0\circ t^1=u^1$. Note that 
\[
(u^1-\delta^0\circ t^1)\circ d^0 = u^1\circ d^0-\delta^0 \circ  t^1 \circ d^0 =\delta^0\circ u^0-\delta^0\circ u^0=0,
\]
so that $u^1-\delta^0\circ t^1$ vanishes on $\rg(d^0)=\ker(d^1)$ and therefore induces a map
\[
\overline{u^1-\delta^0\circ t^1 }\colon E^1/\ker(d^1)\to F^1.
\]
Since $d^1$ is strengthened, we get an inverse morphism $\sigma\colon \rg(d^1)\to E^1/\ker(d^1)$ and since $F^1$ is assumed $\mfh$-relative injective we obtain a map $t^2\colon E^2\to F^1$ making the following diagram commutative:

\[
\xymatrix{
0\ar[r] & \rg(d^0) \ar[d]_\sigma \ar[r]^{\subset} & E^1 \ar@{-->}[ddl]^{t^2}\\
  &E/\ker(d^0) \ar[d]_{\overline{u^1-\delta^0\circ t^1 }} & \\
& F^0 &
} 
\]
(the fact that the inclusion is strengthened follows again since $\rg(d^0)=\ker(d^1)$ and $d^1$ is assumed to be strengthened). Then it follows that
\[
(t^2\circ d^1 +\delta^0 \circ t^1)(x)=\overline{u^1-\delta^0\circ t^1 }(\underbrace{\sigma(d^1(x))}_{=\bar{x}}) +\delta^0 \circ t^1(x) =(u^1-\delta^0 \circ t^1)(x)+\delta^0 t^1 (x)=u^1(x).
\]
This shows how $t^2$ is constructed from $t^1$ and one may now proceed like this by induction.
\end{proof}

\begin{cor}
Consider two strengthened, $\mfh$-relative injective resolutions of the topological $\mfg$-module $E$,
\begin{align*}
&0\To E\To E_1^0 \To E_1^1 \To E_1^2 \To\cdots\\
&0\To E\To \tilde{E}^0 \To \tilde{E}^1 \To \tilde{E}^2 \To\cdots,
\end{align*}
as well as the two induced complexes 
\begin{align*}
K&:0\To \left(E^0\right)^\mfg \To \left(E^1\right)^\mfg \To \left(E^2\right)^\mfg \To\cdots\\
\tilde{K}&:0\To \big(\tilde{E}^0\big)^\mfg \To \big(\tilde{E}^1\big)^\mfg \To \big(\tilde{E}^2\big)^\mfg \To\cdots
\end{align*}
Then
\begin{itemize}
\item[(i)]
The complexes $K$ and $\tilde{K}$ are homotopic through a homotopy consisting of continuous linear maps and hence their cohomology groups are isomorphic as (potentially non-Hausdorff) topological vector spaces.
\item[(ii)] If $u^\bullet , v^\bullet \colon E^\bullet\to \tilde{E}^\bullet$ are two continuous lifts of the identity map $\id\colon E\to E$ then they induce the same morphism $\Cohom^n(K)\to \Cohom^n(\tilde{K})$ for each $n\in \NN_0$.
\end{itemize}
\end{cor}

\begin{proof}
Applying Proposition \ref{prop-fundamental-extension-result} twice, we get morphisms of $\mfg$-complexes $u^\bullet\colon E^\bullet\to \tilde{E}^\bullet$ and $v^\bullet\colon \tilde{E}^\bullet \to E^\bullet$ lifting the identity map on $E$, which therefore induce continuous morphisms of complexes $u_0^\bullet\colon K\to \tilde{K}$ and $v_0^\bullet\colon \tilde{K}\to K $.
By construction, $u^\bullet\circ v^\bullet-\id_{E^\bullet}$ is a lift of the zero map so by  Proposition \ref{prop-fundamental-extension-result} 
this map is homotopic to zero through a homotopy of $\mfg$-morphisms, which therefore induces a homotopy (now consisting of continuous, linear maps) between  $u_0^\bullet\circ v_0^\bullet$ and $\id_{K}$. 
Similarly we get that $v_0^\bullet \circ u_0^\bullet$ is homotopic to $\id_{\tilde{K}}$ proving that $K$ and $\tilde{K}$ are homotopic as claimed. By standard homological algebra, homotopic maps induce the same morphism at the level of cohomology and thus $\Cohom^n(u^\bullet_0)\colon \Cohom^n(K)\to \Cohom^n(\tilde{K})$ and $\Cohom^n(v_0^\bullet)\colon \Cohom(\tilde{K})\to \Cohom(K)$ are continuous and each others inverses.
To see (ii),  take two lifts $u^\bullet ,v^\bullet$ of $\id_E$. Then $u^\bullet-v^\bullet$ is a lift of the zero map and therefore, by Proposition \ref{prop-fundamental-extension-result}, homotopic to zero through a homotopy of $\mfg$-morphisms. Thus $u_0^\bullet-v_0^\bullet \colon K\to \tilde{K}$ is homotopic to zero and  therefore induces the zero map at the level of cohomology. \end{proof}

\section{Smooth vector valued functions}\label{smooth-top-appendix}
In this section we collect the facts on smoothness of functions with values in Frechet spaces that are needed in the proof of Theorem \ref{top-van-est-thm}.
 We denote by $M$  an $n$-dimensional ($2$nd countable) real manifold and by $E$  a Fr{\'e}chet space with a separating family of seminorms $(p_n)_{n\in \NN}$.  
\subsection*{Smooth functions}
Given an open subset $U\subset \RR^n$ and a function $f\colon U\to E$, the usual definition of differentiability, and thus smoothness, of $f$ makes sense (see e.g.~\cite[III,8]{grothendieck-tvs} for more details). More generally, a function $f\colon M\to E$ is called \emph{smooth} at a point $m\in M$ if for some/any local chart $(U,\varphi)$ around $m$, the pull-back $f\circ\varphi^{-1}$ is smooth at $\varphi(m)\in \RR^n$. The set of smooth functions from $M$ to $E$ is denoted $C^\infty(M,E)$.
If $v\in T_mM$ is a tangent vector at $m$ it can be expressed in local coordinates around $m$, and it therefore makes sense to apply $v$ to $f\in C^\infty(M,E)$ and obtain an element $v(f)\in E$. More generally, if $X\in \Vect(M)$ is a smooth vector field on $M$, then $X$ defines an endomorphism $D_X$ of $C^\infty(M,E)$ by setting $D_X(f)(m):=X_m(f)$. Similarly, for any smooth function $f\in C^\infty(M)$ we get an endomorphism $M_f$ of $C^\infty(M,E)$ by setting $M_f(g)(m):= f(m)g(m)$. Denote by $\D$ the subalgebra of $\End(C^\infty(M,E))$ generated by the operators $D_X$ and $M_f$ with  $X\in$Vect$(M)$ and $f\in C^\infty(M)$; the algebra $\D$ is called the \emph{algebra of finite order differential operators}. Then for each compact subset $K\subset M$, each $n\in \NN$ and each $D\in \D$ we get a seminorm $p_{n,D,K}$ on $C^\infty(M,E)$ by setting
$
p_{n,D,K}(f):=\sup_{m\in K} p_n\left(Df(m)\right)
$
and this turns $C^\infty(M,E) $ into a Fr{\'e}chet space (see e.g.~\cite{warner-harmonic-analysis-I} for more details). The topology on $C^\infty(M,E)$ is referred to as the \emph{smooth topology}. \\

\subsection*{Smooth forms}
We will also need the space of vector-valued $p$-forms on $M$. An $E$-valued $p$-form on $M$ is a $C^\infty(M)$-multilinear and alternating map
\[
\omega\colon \underbrace{\Vect(M)\times \cdots \times \Vect(M)}_{\text{$p$ copies}} \To C^\infty(M,E),
\]
where $\Vect(M)$ denotes the $C^\infty(M)$-module of smooth vector fields on $M$. By standard arguments \cite[2.18]{warner}, the value $\omega(X_1,\dots, X_p)(m)$ only depends on $X_1(m),\dots, X_p(m)$ so that $\omega$ can also be thought of as a bundle $(\omega_m)_{m\in M}$ with
$
\omega_m\in \Hom(\wedge^p T_mM, E).
$

\subsection*{Smooth vectors}
If  $G$ is  Lie group with Lie algebra $\mfg$ and $E$ is a continuous $G$-module one defines  a vector $\xi\in E$ to be \emph{smooth} if the function $G\ni g\mapsto g.\xi\in E$ is in $C^\infty(G,E)$, and the set of smooth vectors is denoted $E^\infty$. This is, in general, not a closed subspace in the ambient Fr{\'e}chet space $E$, but it can be endowed with a finer topology for which it is a Fr{\'e}chet space. For this, consider the map
\[
E^\infty\ni \xi \overset{j}{\longmapsto} ( g\mapsto g.\xi)\in C^\infty(G,E).
\]
Then it is not difficult to see that the image of $E^\infty$ under this map is closed in the smooth topology on $C^\infty(G,E)$, and hence $E^\infty$ is a Fr{\'e}chet space when endowed with the subspace topology arising from the identification with a subspace of $C^\infty(M,E)$ via $j$.


\begin{thebibliography}{KPV15}

\bibitem[Bla79]{blanc}
Philippe Blanc.
\newblock Sur la cohomologie continue des groupes localement compacts.
\newblock {\em Ann. Sci. \'Ecole Norm. Sup. (4)}, 12(2):137--168, 1979.

\bibitem[Bou04]{bourbaki-integration}
Nicolas Bourbaki.
\newblock {\em Integration. {I}. {C}hapters 1--6}.
\newblock Elements of Mathematics (Berlin). Springer-Verlag, Berlin, 2004.
\newblock Translated from the 1959, 1965 and 1967 French originals by Sterling
  K. Berberian.

\bibitem[BRS14]{bader-furman-sauer}
Uri Bader, Christian Rosendal, and Roman Sauer.
\newblock On the cohomology of weakly almost periodic group representations.
\newblock {\em J. Topol. Anal.}, 6(2):153--165, 2014.

\bibitem[BW80]{borel-wallach}
Armand Borel and Nolan~R. Wallach.
\newblock {\em Continuous cohomology, discrete subgroups, and representations
  of reductive groups}, volume~94 of {\em Annals of Mathematics Studies}.
\newblock Princeton University Press, Princeton, N.J.; University of Tokyo
  Press, Tokyo, 1980.

\bibitem[CE48]{chevalley-eilenberg}
Claude Chevalley and Samuel Eilenberg.
\newblock Cohomology theory of {L}ie groups and {L}ie algebras.
\newblock {\em Trans. Amer. Math. Soc.}, 63:85--124, 1948.

\bibitem[CE99]{CE}
Henri Cartan and Samuel Eilenberg.
\newblock {\em Homological algebra}.
\newblock Princeton Landmarks in Mathematics. Princeton University Press,
  Princeton, NJ, 1999.
\newblock With an appendix by David A. Buchsbaum, Reprint of the 1956 original.

\bibitem[Gro73]{grothendieck-tvs}
A.~Grothendieck.
\newblock {\em Topological vector spaces}.
\newblock Gordon and Breach Science Publishers, New York-London-Paris, 1973.
\newblock Translated from the French by Orlando Chaljub, Notes on Mathematics
  and its Applications.

\bibitem[Gui80]{guichardet-book}
A.~Guichardet.
\newblock {\em Cohomologie des groupes topologiques et des alg\`ebres de
  {L}ie}, volume~2 of {\em Textes Math\'ematiques [Mathematical Texts]}.
\newblock CEDIC, Paris, 1980.

\bibitem[HM62]{hochschild-mostow}
G.~Hochschild and G.~D. Mostow.
\newblock Cohomology of {L}ie groups.
\newblock {\em Illinois J. Math.}, 6:367--401, 1962.

\bibitem[Kos50]{koszul}
Jean-Louis Koszul.
\newblock Homologie et cohomologie des alg\`ebres de {L}ie.
\newblock {\em Bull. Soc. Math. France}, 78:65--127, 1950.

\bibitem[KP15]{polynomial-cohomology}
David Kyed and Henrik~Densing Petersen.
\newblock Quasi-isometries of nilpotent groups.
\newblock {\em Preprint}, 2015.
\newblock \htmladdnormallink{
  arXiv:1503.04068}{http://http://arxiv.org/abs/1503.04068}.

\bibitem[KPV15]{KPV}
David Kyed, Henrik~Densing Petersen, and Stefaan Vaes.
\newblock {$L^2$}-{B}etti numbers of locally compact groups and their cross
  section equivalence relations.
\newblock {\em Trans. Amer. Math. Soc.}, 367(7):4917--4956, 2015.

\bibitem[KS97]{KS}
Anatoli Klimyk and Konrad Schm{\"u}dgen.
\newblock {\em Quantum groups and their representations}.
\newblock Texts and Monographs in Physics. Springer-Verlag, Berlin, 1997.

\bibitem[Mar06]{florian-martin-reduced-1-cohomology}
Florian Martin.
\newblock Reduced 1-cohomology of connected locally compact groups and
  applications.
\newblock {\em J. Lie Theory}, 16(2):311--328, 2006.

\bibitem[Mos61]{mostow}
G.~D. Mostow.
\newblock Cohomology of topological groups and solvmanifolds.
\newblock {\em Ann. of Math. (2)}, 73:20--48, 1961.

\bibitem[Osb82]{osborn-vector-bundles}
Howard Osborn.
\newblock {\em Vector bundles. {V}ol. 1}, volume 101 of {\em Pure and Applied
  Mathematics}.
\newblock Academic Press, Inc. [Harcourt Brace Jovanovich, Publishers], New
  York, 1982.
\newblock Foundations and Stiefel-{W}hitney classes.

\bibitem[Sha00]{shalom-rigidity}
Yehuda Shalom.
\newblock Rigidity of commensurators and irreducible lattices.
\newblock {\em Invent. Math.}, 141(1):1--54, 2000.

\bibitem[SW99]{schaefer}
H.~H. Schaefer and M.~P. Wolff.
\newblock {\em Topological vector spaces}, volume~3 of {\em Graduate Texts in
  Mathematics}.
\newblock Springer-Verlag, New York, second edition, 1999.

\bibitem[Tes09]{tessera-vanishing}
Romain Tessera.
\newblock Vanishing of the first reduced cohomology with values in an
  {$L^p$}-representation.
\newblock {\em Ann. Inst. Fourier (Grenoble)}, 59(2):851--876, 2009.

\bibitem[vE55]{van-est}
W.~T. van Est.
\newblock On the algebraic cohomology concepts in {L}ie groups. {I},{II}.
\newblock {\em Nederl. Akad. Wetensch. Proc. Ser. A. {\bf 58} = Indag. Math.},
  17:225--233, 286--294, 1955.

\bibitem[War71]{warner}
Frank~W. Warner.
\newblock {\em Foundations of differentiable manifolds and {L}ie groups}.
\newblock Scott, Foresman and Co., Glenview, Ill.-London, 1971.

\bibitem[War72]{warner-harmonic-analysis-I}
Garth Warner.
\newblock {\em Harmonic analysis on semi-simple {L}ie groups. {I}}.
\newblock Springer-Verlag, New York-Heidelberg, 1972.
\newblock Die Grundlehren der mathematischen Wissenschaften, Band 188.

\end{thebibliography}

\def\cprime{$'$} \def\cprime{$'$}

\end{document}